\documentclass[reqno,12pt]{article}
\usepackage{a4wide}
\usepackage{amsmath}
\usepackage{amssymb}
\usepackage{amsthm}
\usepackage{mathrsfs}

\newtheorem{theo}{Theorem}

\newtheorem{prop}{Proposition}

\theoremstyle{remark}

\def\({\left(}
\def\){\right)}
\def\[{\left[}
\def\]{\right]}

\begin{document}

\title{A note on odd zeta values}
\author{Tanguy Rivoal and Wadim Zudilin}

\maketitle

\begin{abstract}
Using a new construction of rational linear forms in odd zeta values and the saddle point method, we prove 
the existence of at least two irrational numbers amongst 
the~33 odd zeta values $\zeta(5)$, $\zeta(7)$, \ldots, $\zeta(69)$.
\end{abstract}

\medskip

\hfill {\em Dedicated to Christian Krattenthaler, on his 60th birthday}

\section{Introduction}\label{sec:intro}

The arithmetic nature of the values of the Riemann zeta function at odd integers is still largely unknown.  
Ap\'ery~\cite{ap} proved that $\zeta(3)$ is an irrational numbers and it was proved in \cite{br, ri1} that infinitely many of the numbers $\zeta(2n+1)$, $n\ge 1$ integer, are irrational and in fact even linearly independent over $\mathbb Q$. The second author proved in \cite{zud1} the existence of at least one irrational number amongst $\zeta(5)$, $\zeta(7)$, $\zeta(9)$, $\zeta(11)$, and in  \cite{zud3} that for any integer $n\ge 0$, there exists at least one irrational number amongst the odd zeta values $\zeta(2n+3)$, $\zeta(2n+5), \ldots, \zeta(16n+7)$ (this result is used below with $n=5)$. Let us also mention that 
Fischler and the second author proved in \cite{fizud} the existence of  two distinct odd integers $m, n \in \{3, 5, \ldots, 139\}$ such that $1, \zeta(m), \zeta(n)$ are linearly independent over $\mathbb Q$.

The goal of this note is to prove the following theorem. 
\begin{theo}\label{theo:2} There exist at least two irrational numbers amongst the \textup{33} odd zeta values $\zeta(5)$, $\zeta(7)$, \ldots, $\zeta(69)$. 
\end{theo}
We shall in fact prove the following equivalent form of Theorem~\ref{theo:2}.
\begin{theo}\label{theo:1} For any integer $m$ such that $1\le m \le 33$, there exists at least one irrational number amongst the \textup{32} odd zeta values $\zeta(2n+3)$, where $n\in\{1,2,\ldots, 33\}\setminus\{m\}$.
\end{theo}
Let us prove the equivalence between both theorems. Applying Theorem~\ref{theo:1} with $m=1$,  let $\zeta(2m_0+3)$ denote one irrational number amongst $\zeta(7)$, $\zeta(9)$, \ldots , $\zeta(69)$. We then apply again Theorem~\ref{theo:1} with $m=m_0$: there exists an  irrational number $\zeta(2n_0+3)$ for some $n_0\in\{1,2,\ldots, 33\}\setminus\{m_0\}$, and Theorem~\ref{theo:2} follows. Conversely, let $\zeta(2m_0+3)$ and $\zeta(2n_0+3)$ be two irrational numbers given by Theorem~\ref{theo:2} with $1\le m_0, n_0\le 33$. Let $m\in \{1,2,\ldots,33\}$: if $m=m_0$, we pick the irrational number $\zeta(2n_0+3)$ while if $m\neq m_0$, we pick the irrational number $\zeta(2m_0+3)$, and Theorem~\ref{theo:1} follows.

We now make some comparisons with  the above mentioned results. The set $\{5, 7, \ldots, 69\}$ is much smaller than $\{3, 5, \ldots, 139\}$, but our  method can't decide if $1$ and the two irrational odd zeta values in Theorem \ref{theo:2} are linearly independent over~$\mathbb Q$. 
Also, there are two irrational odd zeta values $\zeta(m)$ and $\zeta(n)$ with $m\in \{5,7,9,11\}$ and $n\in \{13, 15, \ldots , 87\}$: this bound is worse than in Theorem~\ref{theo:2}, but the  localization is more precise.

In general, results of the type ``there exists an irrational number amongst (\dots)'' are proved using the saddle point method to estimate the decay of certain sequences of linear forms in zeta values. This is the case of the proofs of the results in \cite{ri2, zud1, zud3}. Very recently, the second author introduced a new method to prove in an ``elementary'' way the existence of at least one irrational number amongst  $\zeta(5)$, $\zeta(7)$,  \ldots, $\zeta(25)$. In particular, he completely avoided the use of the saddle point method. 
Our approach is somewhat different, as we shall indeed combine this new approach with the saddle point method  to obtain the proof of Theorem~\ref{theo:1}. Though one can still use an elementary strategy (as outlined in \cite{sprang}) to establish a result similar to Theorem~\ref{theo:2}, this  result is weaker than ours. For further applications of this new method, see \cite{kratzud, sprang}.

The paper is organized as follows. In \S\ref{sec:1}, we introduce some notations and define two series $S_n$ and $\widehat{S}_n$ that enable us to construct ``good'' linear forms in odd zeta values in  \S\ref{sec:2}. In \S\ref{sec:3} we obtain two integrals representations $S_n$ and $\widehat{S}_n$, to which we apply the saddle point method in \S\ref{sec:4}. We complete the proof of Theorem~\ref{theo:1} in \S\ref{sec:5}.

\section{Notations}\label{sec:1}

Let $A$ denote an integer $\ge 15$. For any integer $n\ge0$, we define the rational function of $t$
\begin{align*}
R(t)&:=n!^{A-15}2^{18n}(2t+n)\frac{(t-n)_n^3(t+n+1)_{n}^3(t-n+\frac12)_{3n}^3}{(t)_{n+1}^{A}} 
\\
&= n!^{A-15}2^{-3}(2t+n)\frac{(2t-2n)_{6n+1}^3}{(t)_{n+1}^{A+3}}
\end{align*}
where $(x)_m:=x(x+1)\cdots (x+m-1)$. 

The degree of $R(t)$ is $(15-A)n-A+1\le -2$ so that the partial fraction expansion is 
\begin{equation}\label{eq:5}
R(t)=\sum_{j=1}^A \sum_{m=0}^{n} \frac{p_{j,m}}{(t+m)^j}
\end{equation}
with
\begin{equation}\label{eq:pjm}
p_{j,m}:=\frac{1}{(A-j)!}\big(R(t)(t+m)^A\big)^{(A-j)}_{t=-m}\in  \mathbb Q
\end{equation}
and the residue at $\infty$ of $R(t)$ is $0$.

For any integer $n\ge 0$, we define the series which will be our main objects of study:
$$
S_n:=\sum_{k=1}^\infty R''(k), \qquad \widehat{S}_n:=\sum_{k=1}^\infty  R''\Big(k-\frac12\Big)
$$
where the double prime stands for double differentiation.

A series similar to $S_n$ with $R(t)$ replaced by $n!^{A-6}(2t+n)(t-n)_n^3(t+n+1)_{n}^3(t)_{n+1}^{-A}$ was used in \cite{ri2} to prove the existence of at least one irrational number amongst $\zeta(5)$, $\zeta(7)$, \ldots, $\zeta(21)$. We shall follow the same approach as in \cite{ri2}  and combine it with \cite{zud2} where the additional factor $(t-n+\frac12)_{3n}$ is the principal innovation. Its effects are explained in the comments just after Proposition \ref{prop:1} below.

The arithmetic normalization $n!^{A-15}2^{18n}$ of the rational function $R(t)$ is not optimal in the sense that there is a large factor 
which can be cancelled out from the coefficients $p_{j,m}$. We carefully perform this analysis in Proposition \ref{prop:1}.
It is also possible that the arithmetic behavior of the  coefficients of the linear forms $S_n$ and $\widehat{S}_n$ is even better, because of
certain hypergeometric transformations underlying the construction. This phenomenon, known as ``denominators conjecture'',
is studied in \cite{kratriv}, and the methods developed there might bring in similar improvements for this new
situation. To keep our exposition at an accessible level, we do not include such considerations here.

\section{Construction of two linear forms in odd zeta values} \label{sec:2}

\begin{prop}\label{prop:1} Let us assume that $A\ge 16$ and $n\ge 0$ are both even.

\medskip
$(i)$ We have
\begin{equation}\label{eq:6}
S_n= q_{0,n}+\sum_{\stackrel{j=5}{j\, \textup{odd}}}^{A+1} q_{j,n} \zeta(j)
\end{equation}
and
\begin{equation}\label{eq:7}
\widehat{S}_n= \widehat{q}_{0,n}+\sum_{\stackrel{j=5}{j\, \textup{odd}}}^{A+1} q_{j,n} (2^j-1)\zeta(j)
\end{equation}
where $q_{0,n}$, $\widehat{q}_{0,n}$ and the $q_{j,n}$'s are rational numbers, that depend on $A$.
\medskip

$(ii)$ Let $d_n:=\textup{lcm}\{1,2,\ldots,n\}$ and
$\Phi_n\in \mathbb N$ be the product of prime powers defined in~\eqref{Phi_n} below. Then $\Phi_n^{-3}d_n^{A+2}$ is a common denominator of the coefficients 
$q_{0,n}$, $\widehat{q}_{0,n}$ and the $q_{j,n}$ for odd $j\ge 5$.
\end{prop}
Explicit expressions for the sequences $q_{0,n}$, $\widehat{q}_{0,n}$ and $q_{j,n}$ are given in the proof. The relevance of the series $S_n$ and $\widehat{S}_n$ for Theorem~\ref{theo:1} can now easily be shown: for any odd integer $m\in \{5,7,\ldots, A+1\}$, we have
$$
(2^m-1)S_n-\widehat{S}_n=(2^m-1)q_{0,n}-\widehat{q}_{0,n}+\sum_{\stackrel{j=5, j\neq m}{j\, \textup{odd}}}^{A+1} q_{j,n} (2^m-2^j)\zeta(j)
$$
where $\zeta(m)$ no longer appears.

\begin{proof} $(i)$ Using the partial fraction expansion~\eqref{eq:5}, we get
\begin{align*}
S_n
&=\sum_{j=1}^A \sum_{m=0}^{n} p_{j,m} \sum_{k=1}^\infty \frac{d^2}{dk^2}\frac{1}{(k+m)^j}\\
&=\sum_{j=1}^A \sum_{m=0}^{n} p_{j,m} \sum_{k=1}^\infty \frac{j(j+1)}{(k+m)^{j+2}}\\
&=\sum_{j=1}^A \Big(\sum_{m=0}^{n} p_{j,m}\Big) j(j+1)\zeta(j+2)-\sum_{j=1}^A\sum_{m=0}^{n}\sum_{k=1}^{m} \frac{j(j+1)p_{j,m}}{k^{j+2}}.
\end{align*}
Now, $-\sum_{m=0}^{n} p_{j,1}$ is the residue at $t=\infty$ of $R(t)$, hence is $0$. Moreover, we can use in \eqref{eq:5} the symmetry $R(-n-t)=(-1)^{(A+1)(n+1)}R(t)=-R(t)$ (because $A$ and $n$ are even): it implies that $p_{j,m}=(-1)^{j+1}p_{j,n-m}$. Hence
$
\sum_{m=0}^{n} p_{j,m}=0
$
when $j$ is even. Therefore
$$
S_n=\sum_{\stackrel{j=5}{j\, \textup{odd}}}^{A+1} \Big(\sum_{m=0}^{n} (j-2)(j-1)p_{j-2,m}\Big) \zeta(j)-\sum_{j=1}^A\sum_{m=0}^{n}\sum_{k=1}^{m} \frac{j(j+1)p_{j,m}}{k^{j+2}}.
$$

We now let $\widehat{\zeta}(s):=\sum_{k=1}^\infty \frac{1}{(k-\frac12)^s}=(2^s-1)\zeta(s)$. 
We also have
\begin{align}
\widehat{S}_n
&=\sum_{j=1}^A \sum_{m=0}^{n} p_{j,m} \sum_{k=1}^\infty \frac{d^2}{dk^2}\frac{1}{(k+m-\frac12)^j}\notag
\\
&=\sum_{j=1}^A \sum_{m=0}^{n} p_{j,m} \sum_{k=1}^\infty \frac{j(j+1)}{(k+m-\frac12)^{j+2}}\notag
\\
&=\sum_{j=1}^A \Big(\sum_{m=0}^{n} j(j+1)p_{j,m}\Big) \widehat{\zeta}(j+2)-\sum_{j=1}^A\sum_{m=0}^{n}\sum_{k=1}^{m} \frac{j(j+1)p_{j,m}}{(k-\frac12)^{j+2}}. \label{eq:9}
\end{align}
Hence
$$
\widehat{S}_n=\sum_{\stackrel{j=5}{j\, \textup{odd}}}^{A+1} \Big(\sum_{m=0}^{n} (j-2)(j-1)p_{j-2,m}\Big) (2^j-1)\zeta(j)-\sum_{j=1}^A\sum_{m=0}^{n}\sum_{k=1}^{m} \frac{j(j+1)p_{j,m}}{(k-\frac12)^{j+2}}.
$$

We have thus proved \eqref{eq:6} and \eqref{eq:7} with
\begin{equation}\label{eq:10}
q_{j,n}:=\sum_{m=0}^{n} (j-2)(j-1)p_{j-2,m}, \quad j\; \textup{odd} \ge 5,
\end{equation}
\begin{equation}\label{eq:11}
q_{0,n}:=-\sum_{j=1}^A\sum_{m=0}^{n}\sum_{k=1}^{m} \frac{j(j+1)p_{j,m}}{k^{j+2}}
\end{equation}
and
\begin{equation}\label{eq:8}
\widehat{q}_{0,n}:=-\sum_{j=1}^A\sum_{m=0}^{n}\sum_{k=1}^{m} \frac{j(j+1)p_{j,m}}{(k-\frac12)^{j+2}}.
\end{equation}

$(ii)$ We can apply {\em mutatis mutandis}  \cite[Lemma 1]{zud2} and get that $d_n^{A-j}p_{j,m}\in \mathbb Z$ for any $j\in\{1,\ldots,A\}$ and any $m \in\{0, \ldots, n\}$. 

On the other hand, observe that the rational function $R(t)$ can be written as $(2t+\nobreak n)\*F(t)^3G(t)^{A-15}$
where
$$
F(t):=2^{6n}\frac{(t-n)_n(t+n+1)_{n}(t-n+\frac12)_{3n}}{(t)_{n+1}^5}
\quad\text{and}\quad
G(t):=\frac{n!}{(t)_{n+1}}.
$$
The numbers
$$
\big(F(t)(t+m)^5\big)_{t=-m}=\frac{(2n+2m)!\,(4n-2m)!}{m!^6(n-m)!^6}
\quad\text{for}\; m=0,1,\dots,n,
$$
have a large common factor. Indeed, by the standard formula for $p$-adic valuation $v_p(n!)=\sum_{\ell=1}^\infty [n/p^\ell]$ of $n!$, we get
$$
v_p\Big(\frac{(2n+2m)!\,(4n-2m)!}{m!^6(n-m)!^6}\Big)
=\rho\Big(\frac np,\frac mp\Bigr)
\ge\rho_0\Big(\frac np\Bigr)
$$
for an odd prime $p>2\sqrt{n}$, where the function
\begin{multline*}
\rho(x,y):=\lfloor2x+2y\rfloor+\lfloor4x-2y\rfloor-6\lfloor y\rfloor-6\lfloor x-y\rfloor
\\
=6\{y\}+6\{x-y\}-\{2x+2y\}-\{4x-2y\}
\end{multline*}
is periodic of period 1 both in $x$ and $y$, and 
$$
\rho_0(x):=\min_{y\in\mathbb R}\rho(x,y)=\min_{0\le y<1}\rho(x,y)
$$
is also 1-periodic. The latter can be explicitly given on the interval $0\le x<1$ by
$$
\rho_0(x)=\begin{cases}
0 & \text{if $0\le x<\frac13$}, \\
1 & \text{if $\frac13\le x<\frac12$}, \\
2 & \text{if $\frac12\le x<\frac23$}, \\
3 & \text{if $\frac23\le x<\frac56$}, \\
4 & \text{if $\frac56\le x<1$}
\end{cases}
$$
(see, for example, \cite[\S\,4]{zud3} for similar  computations). Thus, denoting
\begin{equation}
\Phi_n:=\prod_{2\sqrt n<p\le n}p^{\rho_0(n/p)}
\label{Phi_n}
\end{equation}
(a product over prime numbers), we obtain
$$
\Phi_n^{-1}\big(F(t)(t+m)^5\big)_{t=-m}\in\mathbb Z
\quad\text{for}\; m=0,1,\dots,n.
$$
Following the lines of the proof of  \cite[Lemma 4.2]{zud3}, we see that these inclusions imply
$$
\Phi_n^{-1}d_n^k\,\frac1{k!}\big(F(t)(t+m)^5\big)^{(k)}_{t=-m}\in\mathbb Z
$$
for $m=0,1,\dots,n$ and all integers $k\ge 0$. For the same range of indices we also have
$$
d_n^k\,\frac1{k!}\big(G(t)(t+m)\big)^{(k)}_{t=-m}\in\mathbb Z.
$$
Combining these inclusions with the representation $R(t)=(2t+\nobreak n)\*F(t)^3G(t)^{A-15}$ and using  Leibniz's rule
for differentiating products, we conclude that the numbers $p_{j,m}$ in \eqref{eq:pjm} satisfy
$\Phi_n^{-3}d_n^{A-j}p_{j,m}\in\mathbb Z$ for any $j\in\{1,\ldots,A\}$ and any $m \in\{0, \ldots, n\}$.

From the expressions of  $q_{j,n}$ in \eqref{eq:10} and \eqref{eq:11}, we then deduce that $\Phi_n^{-3}d_n^{A+2}q_{0,n}\in \mathbb Z$ and  $\Phi_n^{-3}d_n^{A+2}q_{j,n}\in \mathbb Z$ for any odd $j\ge 5$.

\medskip

This argument does not work directly for $\widehat{q}_{0,n}$ with the expression in \eqref{eq:8} because it leads only to the weaker estimate $\Phi_n^{-3}d_{2n}^{A+2}\widehat{q}_{0,n}\in \mathbb Z$.
We need an alternative expression for $\widehat{q}_{0,n}$, and for this we follow \cite[p. 5]{zud2}. We let $\omega=\lfloor \frac{n-1}{2}\rfloor$. Since $R''(k)=0$ for $k=-\frac12, -\frac{3}{2}, \ldots, -n+\frac{1}{2}$, in particular $R''(k-\frac12)=0$ for $k=-\omega, \ldots, -1, 0$ when~(\footnote{This not true when $n=0$, but in this case $\widehat{q}_{0,n}=0$ and there is nothing to prove.}) $n\ge 1$. Hence, for any $n\ge 1$,  
\begin{align}
\widehat{S}_n &=\sum_{k=-\omega}^\infty R''\Big(k-\frac12\Big)=\sum_{j=1}^A \sum_{m=0}^{n} p_{j,m} \sum_{k=-\omega}^\infty \frac{d^2}{dk^2}\frac{1}{(k+m-\frac12)^j}\notag
\\
&=\sum_{j=1}^A \sum_{m=0}^{\omega} p_{j,m} \sum_{k=-\omega}^\infty \frac{j(j+1)}{(k+m-\frac12)^{j+2}}
+\sum_{j=1}^A \sum_{m=\omega+1}^{n} p_{j,m} \sum_{k=-\omega}^\infty \frac{j(j+1)}{(k+m-\frac12)^{j+2}} \notag
\displaybreak[2]\\
&=\sum_{j=1}^A \sum_{m=0}^{\omega} p_{j,m} \Big(\sum_{k=m-\omega}^0 \frac{j(j+1)}{(k-\frac12)^{j+2}} +\sum_{k=1}^\infty \frac{j(j+1)}{(k-\frac12)^{j+2}} \Big) \notag
\\
& \qquad \qquad + \sum_{j=1}^A \sum_{m=\omega+1}^{n} p_{j,m} \Big(\sum_{k=1}^\infty \frac{j(j+1)}{(k-\frac12)^{j+2}} - \sum_{k=1}^{m-\omega-1}\frac{j(j+1)}{(k-\frac12)^{j+2}}\Big)
\notag
\displaybreak[2]\\
&=\sum_{j=1}^{A}\Big(\sum_{m=0}^{n} j(j+1)p_{j,m}\Big)\widehat{\zeta}(j+2) + \sum_{j=1}^A \sum_{m=1}^\omega (-1)^jj(j+1)p_{j,m}\sum_{k=0}^{\omega-m}\frac{1}{(k+\frac12)^{j+2}} \notag
\\
& \qquad \qquad -\sum_{j=1}^A \sum_{m=\omega+1}^n (-1)^jj(j+1)p_{j,m}\sum_{k=1}^{m-\omega-1}\frac{1}{(k-\frac12)^{j+2}}. \label{eq:12}
\end{align}
Comparing \eqref{eq:12} with \eqref{eq:9}, we see that when $n\ge 1$, 
\begin{equation*}
\widehat{q}_{0,n}=\sum_{j=1}^A \sum_{m=1}^\omega p_{j,m}\sum_{k=0}^{\omega-m}\frac{(-1)^jj(j+1)}{(k+\frac12)^{j+2}}
 -\sum_{j=1}^A \sum_{m=\omega+1}^n p_{j,m}\sum_{k=1}^{m-\omega-1}\frac{(-1)^jj(j+1)}{(k-\frac12)^{j+2}}. \qquad \quad
\end{equation*}
This expression is  more suitable than \eqref{eq:8} and an analysis similar to \cite[p. 6]{zud2} shows that $\Phi_n^{-3}d_n^{A+2}\widehat{q}_{0,n}\in \mathbb Z$ as expected.
\end{proof}

\section{Complex integral representations of $S_n$ and $\widehat S_n$} \label{sec:3}
In this section, the integers $A\ge 15$ and $n\ge 0$ are not necessarily even.

\begin{prop}\label{prop:2} Let $L$ denotes any vertical line $\{c+iy, y\in \mathbb R\}$ oriented from $y>0$ to $y<0$, where $c\in (\frac12,n)$. We have
\begin{equation}\label{eq:1}
S_n=\frac{n!^{A-15}}{i\pi}\int_L (2t+n)\frac{\Gamma(t)^{A+3}\Gamma(2t+4n+1)^3\Gamma(2n-2t+1)^3}{\Gamma(t+n+1)^{A+3}}\cos(\pi t)^4dt
\end{equation}
and
\begin{equation}\label{eq:2}
\widehat{S}_n=-\frac{n!^{A-15}}{i\pi}\int_L (2t+n-1)\frac{\Gamma(t-\frac{1}{2})^{A+3}\Gamma(2t+4n)^3\Gamma(2n-2t+2)^3}{\Gamma(t+n+\frac12)^{A+3}}\cos(\pi t)^4dt.
\end{equation}
\end{prop}
\begin{proof} We adapt the proof of \cite[Lemma 4]{ri2}. Using the trivial identity $(x)_m=\Gamma(x+m)/\Gamma(x)$, we first rewrite
\begin{align}
R(t)&=n!^{A-15}2^{-3} (2t+n)\frac{\Gamma(t)^{A+3}\Gamma(2t+4n+1)^3}{\Gamma(t+n+1)^{A+3}\Gamma(2t-2n)^3} \label{eq:R1}
\\
&=n!^{A-15}2^{-3} (2t+n)\frac{\Gamma(t)^{A+3}\Gamma(2t+4n+1)^3\Gamma(2n-2t+1)^3}{\Gamma(t+n+1)^{A+3}}\cdot\frac{\sin(2\pi t)^3}{\pi^3}\label{eq:R}
\end{align}
where \eqref{eq:R} is a consequence of the reflection formula $\Gamma(x)\Gamma(1-x)=\pi/\sin(\pi x)$ applied to \eqref{eq:R1} with $x=2t-2n$.

Let $u$ be such that $\textup{Re}(u)\le 0$ and $\vert \textup{Im}(u)\vert\le 3\pi$. Let $c$ denote any real number in $(\frac12,n)$, and $N$ any integer $\ge n+1$. Let $C_N$ denote the rectangular contour (oriented positively) with sides $[c-iN, N+\frac12-iN]$, $[N+\frac12-iN, N+\frac12+iN]$, $[N+\frac12+iN, c+iN]$, $[c+iN,c-iN]$. 
By the residue theorem, and because $R^{(j)}(k)=0$ for $k=1, \ldots, n$ and $j\in\{0,1,2\}$, we have
\begin{align*}
\frac{1}{2i\pi}\int_{C_N} R(t) \frac{\pi^3}{\sin(\pi t)^3}e^{u t}dt=
\sum_{k=n+1}^N \left(\frac{\pi ^2+u^2}{2}R(k)+ uR'(k)+ \frac12R''(k)\right)(-e^u)^k.
\end{align*}
The conditions on $u$ ensure that, as $N\to +\infty$, we have 
\begin{equation}\label{eq:17}
\frac{1}{2i\pi}\int\limits_{c+i\infty}^{c-i\infty} R(t) \frac{\pi^3}{\sin(\pi t)^3}e^{u t}dt =
\sum_{k=n+1}^\infty \left(\frac{\pi ^2+u^2}{2}R(k)+ uR'(k)+ \frac12R''(k)\right)(-e^u)^k.
\end{equation}
In particular, summing the two specializations of \eqref{eq:17} for $u=i\pi$ and $u=-i\pi$, we get
$$
\frac{1}{i\pi}\int\limits_{c+i\infty}^{c-i\infty} R(t) \frac{\pi^3 \cos(\pi t)}{\sin(\pi t)^3}dt=S_n. 
$$
Hence, 
\begin{align*}
S_n
&=\frac{n!^{A-15}2^{-3}}{i\pi}\int\limits_{c+i\infty}^{c-i\infty} (2t+n)\frac{\Gamma(t)^{A+3}\Gamma(2t+4n+1)^3\Gamma(2n-2t+1)^3}{\Gamma(t+n+1)^{A+3}}\cdot\frac{\sin(2\pi t)^3\cos( \pi t)}{\sin(\pi t)^3} dt
\\
&=\frac{n!^{A-15}}{i\pi}\int\limits_{c+i\infty}^{c-i\infty} (2t+n)\frac{\Gamma(t)^{A+3}\Gamma(2t+4n+1)^3\Gamma(2n-2t+1)^3}{\Gamma(t+n+1)^{A+3}}\cos( \pi t)^4 dt.
\end{align*}

With $t-\frac12$ instead of $t$ in~\eqref{eq:R}, we have
\begin{align*}
R\Big(t-\frac12\Big)
&=-n!^{A-15}2^{-3} (2t+n-1)\frac{\Gamma(t-\frac12)^{A+3}\Gamma(2t+4n)^3\Gamma(2n-2t+2)^3}{\Gamma(t+n+\frac12)^{A+3}}\cdot\frac{\sin(2\pi t)^3}{\pi^3}.
\end{align*}
Since $R^{(j)}(k-\frac12)=0$ for $k=1, \ldots, n$ and $j\in\{0,1,2\}$, it follows again that
$$
\frac{1}{i\pi}\int_{C_N} R\Big(t-\frac12\Big) \frac{\pi^3 \cos(\pi t)}{\sin(\pi t)^3}dt = \sum_{k=1}^N R''\Big(k-\frac12\Big).
$$ 
Hence, upon letting $N\to +\infty$, 
$$
\widehat{S}_n=-\frac{n!^{A-15}}{i\pi}\int\limits_{c+i\infty}^{c-i\infty} (2t+n-1)\frac{\Gamma(t-\frac{1}{2})^{A+3}\Gamma(2t+4n)^3\Gamma(2n-2t+2)^3}{\Gamma(t+n+\frac12)^{A+3}}\cos(\pi t)^4dt
$$
and the proof is now complete.
\end{proof}

\section{Asymptotic behavior of $S_n$ and $\widehat{S}_n$}\label{sec:4}

To prove Theorem~\ref{theo:2}, we have to  determine $A$ even minimal such that $\Phi_{n}^{-3}d_{n}^{A+2}S_{n}\to 0$ and $\Phi_{n}^{-3}d_{n}^{A+2}\widehat{S}_{n}\to 0$ along a subsequence of the even integers. It turns out that $A=68$ is this minimal value. 

\begin{prop} \label{prop:3} Let $A=68$. There exists an increasing sequence of even integers $\sigma(n)$ such that 
$$
\lim_{n\to +\infty} \vert S_{\sigma(n)}\vert^{1/\sigma(n)}= 
\lim_{n\to +\infty} \vert \widehat{S}_{\sigma(n)}\vert^{1/\sigma(n)}= e^{-\kappa}
$$
with $\kappa \approx 66.1727$, 
and 
$$
\lim_{n\to +\infty} \frac{\widehat{S}_{\sigma(n)}}{S_{\sigma(n)}}=-1.
$$
\end{prop}

\begin{proof} We let the integer $A\ge 15$ be unspecified for the moment, and $n$ is any integer $\ge 0$. 
We define $\log(z)=\ln\vert z\vert+i\arg(z)$, with $\vert \arg(z)\vert<\pi$. Stirling's formula reads:
$$
\Gamma(z)=z^z e^{-z}\sqrt{\frac{2\pi}z} \Big(1+\mathcal{O}\big(\frac1z\big)\Big)
$$
as $z\to \infty$ in any angular sector $\vert \arg(z)\vert\le \pi-\varepsilon$, and $\mathcal{O}$ depends on $\varepsilon>0$. Let us define  
\begin{multline*}
\quad \varphi_0(t):=(A+3)t\log(t)+(6t+12)\log(2t+4)\\
+(6-6t)\log(2-2t)-(A+3)(t+1)\log(t+1),\quad
\end{multline*}
where all the logarithms are defined with their principal determinations. This function is analytic in $\Omega:=\mathbb C\setminus \big((-\infty,0] \cup [1,+\infty)\big)$.

After changing $t$ to $tn$ in~\eqref{eq:1} and \eqref{eq:2} in Proposition \ref{prop:2}, we apply Stirling's formula to the various Gamma functions of the integrands and get
\begin{equation}\label{eq:3}
S_n= \frac{2i(2\pi)^{(A-9)/2}}{n^{(A+9)/2}}\int\limits_{c-i\infty}^{c+i\infty} g_n(t) f(t)^n \cos(\pi n t)^4 dt
\end{equation}
and 
\begin{equation}\label{eq:4}
\widehat{S}_n= 
\frac{-2i(2\pi)^{(A-9)/2}}{n^{(A+9)/2}}
\int\limits_{c-i\infty}^{c+i\infty} \widehat{g}_n(t) f(t)^n \cos(\pi n t)^4 dt
\end{equation}
where $c$ is any real number in $(0,1)$, and 
$$
f(t)=\exp(\varphi_0(t))=\frac{t^{t(A+3)}(2t+4)^{3(2t+4)}(2-2t)^{3(2-2t)}}{(t+1)^{(t+1)(A+3)}}
$$
is analytic in  $\Omega$. 
The functions $g_n(t)$ and $\widehat{g}_n(t)$ are analytic in $\Omega$ and such that as $n\to +\infty$
$$
g_n(t)=g(t)\Big(1+\mathcal{O}\Big(\frac 1n\Big)\Big),
$$
$$
\widehat{g}_n(t)=g(t)\frac{(t+1)^{(A+3)/2}(2-2t)^3}{t^{(A+3)/2}(2t+4)^3}\Big(1+\mathcal{O}\Big(\frac 1n\Big)\Big).
$$
with 
$$
g(t)=\frac{(2t+1)(2t+4)^{3/2}(2-2t)^{3/2}}
{t^{(A+3)/2}(t+1)^{(A+3)/2}}.
$$
For any $\varepsilon>0$, when $t\in \{\vert\arg(t)\vert\le \pi-\varepsilon \} \cap \{\vert\arg(1-t)\vert\le \pi-\varepsilon \}$,  the constants in the  $\mathcal{O}(\frac 1n)$ depends on $\varepsilon$ and not on $t$.  

\medskip

We shall now apply the saddle point method to estimate each integral in \eqref{eq:3} and \eqref{eq:4}. Since 
$$
\cos(x)^4= \frac38\big(\cos(4x) + 4\cos(2x) + 3\big),
$$
we have
\begin{equation}\label{eq:16}
S_n=\frac{2i(2\pi)^{(A-9)/2}}{n^{(A+9)/2}}\sum_{\ell=-2}^2 u_\ell J_{\ell,n} \quad
\textup{and} \quad
\widehat{S}_n=-\frac{2i(2\pi)^{(A-9)/2}}{n^{(A+9)/2}}\sum_{\ell=-2}^2 u_\ell \widehat{J}_{\ell,n}
\end{equation}
where
$$
J_{\ell,n} = \int\limits_{c-i\infty}^{c+i\infty} g_n(t) f(t)^n e^{2\ell i\pi n t}dt \quad
\textup{and} \quad
\widehat{J}_{\ell,n} = \int\limits_{c-i\infty}^{c+i\infty} \widehat{g}_n(t) f(t)^n e^{2\ell i\pi n t}dt,
$$
and $u_{-4}=u_4=\frac{3}{16}$, $u_{-2}=u_2=\frac{3}{4}$ and $u_0=\frac{9}{16}$.

Observe that $\overline{J_{\ell,n}}=-J_{-\ell,n}$ and $\overline{\widehat{J}_{\ell,n}}=-\widehat{J}_{-\ell,n}$.

\medskip

We now give the details necessary to complete the proof when $A=68$. For the statement of the saddle point method, see \cite[Lemma 3]{ri2}.

\medskip

\noindent $\bullet$ Case $\ell=0$. 
 The equation $\varphi_0'(t)=0$ admits a solution
$t_0\approx 0.9991\in(0,1)$
and we have $g(t_0)=\widehat{g}(t_0)\neq 0$ and $\varphi_0''(t_0)\neq 0$ (see below). The real function $f_0(y):=\textup{Re}(\varphi_0)(t_0+iy)$ admits a unique maximum at $y=0$; this is proved by studying the sign of 
\begin{align}
f_0'(y)&=-\textup{Im}(\varphi_0')(t_0+iy) \notag
\\
&=-71\arg(t_0+iy)-6\arg(2t_0+2iy+4) \notag
\\&\qquad \qquad
+6\arg(2-2t_0-2iy)+71\arg(1+t_0+iy). \label{eq:13}
\end{align}
This is done by the same method as in \cite[Lemma 4]{ri2},  using
$\arg(z)=\arctan(\textup{Im}(z)/\textup{Re}(z))$ when $\textup{Re}(z)>0$ (which is the case of all the various values of $z$ in \eqref{eq:13}). Hence, shifting the line $c+i\mathbb R$ to the line $t_0+i\mathbb R$ (oriented from $\textup{Im}(t)<0$ to $\textup{Im}(t)>0$), we can apply the saddle point method to $J_{0,n}$ and $\widehat{J}_{0,n}$, and get
$$
J_{0,n}= ig(t_0) \sqrt{\frac{2\pi}{n\varphi_0''(t_0)} } f(t_0)^n \big(1+o(1)\big) \quad \textup{and} \quad 
\widehat{J}_{0,n}= i \widehat{g}(t_0) \sqrt{\frac{2\pi}{n\varphi_0''(t_0)} } f(t_0)^n \big(1+o(1)\big).
$$
Numerically, 
$$
t_0\approx 0.99918,
\quad
f(t_0)\approx 1.8127 \times 10^{-29}  
$$
$$
g(t_0)=\frac{(2t_0+1)(2t_0+4)^{3/2}(2-2t_0)^{3/2}}
{t_0^{71/2}(t_0+1)^{71/2}}\approx 6.2647 \times 10^{-14}
$$
$$
\widehat{g}(t_0)=g(t_0)\frac{(t_0+1)^{71/2}(2-2t_0)^3}{t_0^{71/2}(2t_0+4)^3}=g(t_0), \quad \varphi_0''(t_0)\approx 7373.2123
$$

\noindent $\bullet$ Case $\ell=1$. Let us define  $\varphi_1(t):=\varphi_0(t)+2i\pi t$. The equation $\varphi_1'(t)=0$ admits a solution 
$t_1\approx 0.9995-i0.0007$.
We have $g(t_1)=-\widehat{g}(t_1)\neq 0$ and $\varphi_1''(t_1)=\varphi_0''(t_1)\neq 0$ (see below).
The real function $f_1(y):=\textup{Re}(\varphi_1)(t_1+iy)$ admits a unique maximum at $y=0$; again, this is proved by studying the sign of 
\begin{align*}
f_1'(y)&=-\textup{Im}(\varphi_1')(t_0+iy)=-\textup{Im}(\varphi_0')(t_0+iy)-2\pi
\end{align*}
as in \cite[Lemma 4]{ri2}. Hence, shifting the line $c+i\mathbb R$ to the line $t_1+i\mathbb R$ (oriented from $\textup{Im}(t)<0$ to $\textup{Im}(t)>0$), we can apply the saddle point method to $J_{1,n}$ and $\widehat{J}_{1,n}$, and get
$$
J_{1,n}= ig(t_1) \sqrt{\frac{2\pi}{n\varphi_0''(t_1)} } f(t_1)^n \big(1+o(1)\big)
\quad \textup{and} \quad 
\widehat{J}_{1,n}= i \widehat{g}(t_1) \sqrt{\frac{2\pi}{n\varphi_0''(t_1)} } f(t_1)^n \big(1+o(1)\big).
$$
Numerically,
$$
t_1\approx 0.9995-i0.0007, 
$$
$$f(t_1)e^{2i\pi t_1}\approx 1.8171 \times 10^{-29}-i7.7425 \times 10^{-32}
$$
$$
g(t_1)=\frac{(2t_1+1)(2t_1+4)^{3/2}(2-2t_1)^{3/2}}
{t_1^{71/2}(t_1+1)^{71/2}}\approx -i1.8629 \times 10^{-15}+6.1544 \times 10^{-14}
$$
$$
\widehat{g}(t_1)=g(t_1)\frac{(t_1+1)^{71/2}(2-2t_1)^3}{t_1^{71/2}(2t_1+4)^3}=-g(t_1)
$$
$$\varphi_0''(t_1)\approx 3724.1063-i6320.4884
$$

\noindent $\bullet$ Case $\ell=-1$.  We have 
$$
J_{-1,n}= -\overline{J_{1,n}}=i\overline{g(t_{1})} \sqrt{\frac{2\pi}{n\overline{\varphi_{0}''(t_{1})}} } \overline{f(t_{1})}^n \big(1+o(1)\big)
$$
and 
$$
\widehat{J}_{1,n}=-\overline{\widehat{J}_{1,n}} =i \overline{\widehat{g}(t_{1})} \sqrt{\frac{2\pi}{n\overline{\varphi_{0}''(t_{1})}} } \overline{f(t_{1})}^n \big(1+o(1)\big).
$$

\noindent $\bullet$ Case $\ell=2$. Let us define  $\varphi_2(t):=\varphi_0(t)+4i\pi t$. The equation $\varphi_2'(t)=0$ admits a solution 
$t_2\approx 1.0004-i0.0007$. 
We have $g(t_2)=\widehat{g}(t_2)\neq 0$ and $\varphi_2''(t_2)=\varphi_0''(t_2)\neq 0$ (see below). 
The situation is a different here because $\textup{Re}(t_2)>1$ and thus the straightline $t_2+i\mathbb R$ is not contained in $\Omega$. Instead, we define the polygonal contour $L:=\{t_2+y, y\le 0\}\cup [t_2,1] \cup \{1+iy, y\ge 0\}$. One can then prove that when $t$ varies in $L$, $\textup{Re}(\varphi_2)(t) $ has a unique maximum at $t=t_2$; see the proof of \cite[Lemma 10]{firi} or \cite{zud3} for similar computations. We now replace the ``corner'' at $t=1$ by an arc of circle centered at $1$ and of small radius $\eta>0$ (and contained in $\Omega$), and connect it to the remaining parts of $[t_2,1]$ and $\{1+iy, y\ge 0\}$: we obtain a path $L'$ along which $\textup{Re}(\varphi_2)(t) $ still has a unique maximum at $t=t_2$ provided $\eta$ is small enough, and $L'$ is at positive distance of $(-\infty,0] \cup [1,+\infty)$.
Hence, moving the line $c+i\mathbb R$ to $L'$ (oriented from $\textup{Im}(t)<0$ to $\textup{Im}(t)>0$) we can apply the saddle point method to $J_{2,n}$ and $\widehat{J}_{2,n}$, and get
$$
J_{2,n}= ig(t_2) \sqrt{\frac{2\pi}{n\varphi_0''(t_2)} } f(t_2)^n \big(1+o(1)\big)\quad \textup{and} \quad 
\widehat{J}_{2,n}= i \widehat{g}(t_2) \sqrt{\frac{2\pi}{n\varphi_0''(t_2)} } f(t_2)^n \big(1+o(1)\big).
$$
Numerically,
$$
t_2\approx 1.0004-i0.0007,
$$
$$f(t_2)e^{4i\pi t_2}\approx 1.8261 \times 10^{-29}-i7.8209 \times 10^{-32}
$$
$$
g(t_2)=\frac{(2t_2+1)(2t_2+4)^{3/2}(2-2t_2)^{3/2}}
{t_2^{71/2}(t_2+1)^{71/2}}\approx -5.9420 \times 10^{-14}-i1.8161 \times 10^{-15}
$$
$$
\widehat{g}(t_2)=g(t_2)\frac{(t_2+1)^{71/2}(2-2t_2)^3}{t_2^{71/2}(2t_2+4)^3}=g(t_2)
$$
$$
\varphi_0''(t_2)\approx -3574.1082-i6320.4861
$$

\noindent $\bullet$ Case $\ell=-2$. We have 
$$
J_{-2,n}= -\overline{J_{2,n}}=i\overline{g(t_{2})} \sqrt{\frac{2\pi}{n\overline{\varphi_{0}''(t_{2})}} } \overline{f(t_{2})}^n \big(1+o(1)\big)
$$
and 
$$
\widehat{J}_{-2,n}= -\overline{\widehat{J}_{2,n}}=i \overline{\widehat{g}(t_{2})} \sqrt{\frac{2\pi}{n\overline{\varphi_{0}''(t_{2})}} } \overline{f(t_{2})}^n \big(1+o(1)\big).
$$

We now observe that $\vert f(t_2)e^{4i\pi t_2}\vert > \vert f(t_1)e^{2i\pi t_1}\vert > \vert f(t_0)\vert$. By an argument similar to \cite[Lemma 5]{ri2},  we deduce from \eqref{eq:16} and the above asymptotic estimates for $J_{2,n}$, $J_{-2,n}$, $\widehat{J}_{2,n}$ and $\widehat{J}_{-2,n}$ the existence of an increasing sequence of even integers $\sigma(n)$ such that 
$$
\lim_{n\to +\infty} \vert S_{\sigma(n)}\vert^{1/\sigma(n)}= \vert f(t_2)e^{4i\pi t_2}\vert, \quad 
\lim_{n\to +\infty} \vert \widehat{S}_{\sigma(n)}\vert^{1/\sigma(n)}= \vert f(t_2)e^{4i\pi t_2}\vert
$$
and moreover
$$
\lim_{n\to +\infty} \frac{
\widehat{S}_{\sigma(n)}}{S_{\sigma(n)}}=-1.
$$
We conclude the proof with the estimate 
$
\vert f(t_2)e^{4i\pi t_2}\vert \approx \exp(-66.1727).
$
\end{proof}

\section{Proof of Theorem \ref{theo:1}}\label{sec:5}
We first remark that the asymptotic behaviors of $d_n$ and of $\Phi_n$  both follow from the prime number theorem:
$$
\lim_{n\to\infty}\frac{\log(d_n)}n=1
\quad\text{and}\quad
\delta:=\lim_{n\to\infty}\frac{\log(\Phi_n)}n=\int_0^1\rho_0(t)\,d(\psi(t)+t^{-1})
\approx 1.29564, 
$$
where $\psi(t):=\Gamma'(t)/\Gamma(t)$ and the function $\rho_0(t)$ is defined in the proof of Proposition \ref{prop:1}.

We now take $A=68$ in Proposition \ref{prop:1}. 
For any odd integer $m\in \{5,7,\ldots, 69\}$ and any even integer $n$, we have
$$
\Phi_n^{-3}d_n^{70}\big((2^m-1)S_n-\widehat{S}_n\big)=Q_{0,n}+\sum_{\stackrel{j=5, j\neq m}{j\, \textup{odd}}}^{69} Q_{j,n} \zeta(j).
$$
with $Q_{0,n} := \Phi_n^{-3}d_n^{70}\big((2^m-1)q_{0,n}-\widehat{q}_{0,n}\big) \in \mathbb Z$ and $Q_{j,n}=\Phi_n^{-3}d_n^{70}(2^m-2^j)q_{j,n} \in \mathbb Z$.

By Proposition \ref{prop:3}, we have  
\begin{align*}
\lim_{n\to +\infty} &\Big\vert \Phi_{\sigma(n)}^{-3}d_{\sigma(n)}^{70}\big((2^m-1)S_{\sigma(n)}-\widehat{S}_{\sigma(n)}\big) \Big\vert^{1/\sigma(n)}\\
&=\lim_{n\to +\infty} \big\vert \Phi_{\sigma(n)}^{-3}d_{\sigma(n)}^{70}S_{\sigma(n)} \big\vert^{1/\sigma(n)} \times\lim_{n\to +\infty} \bigg\vert (2^m-1)-\frac{\widehat{S}_{\sigma(n)}}{S_{\sigma(n)}} \bigg\vert^{1/\sigma(n)} 
\\&
= e^{70-\kappa-3\delta} \approx e^{-0.0597} \in (0,1).
\end{align*}
This proves the existence of at least one irrational number amongst the odd zeta values $\zeta(j)$ with $j$ odd in $\{5,7, \ldots, 69\}\setminus\{m\}$.

\def\refname{Bibliography}

\bigskip

\noindent Tanguy Rivoal,  Institut Fourier, CNRS et Universit\'e Grenoble Alpes,  CS 40700, 38058 Grenoble cedex 9, France

\medskip
\noindent tanguy.rivoal [\&] univ-grenoble-alpes.fr

\bigskip 

\noindent Wadim Zudilin, Institute for Mathematics, Astrophysics and Particle Physics,
Radboud Universiteit, PO Box 9010,
6500~GL Nijmegen, The Netherlands

\medskip
\noindent w.zudilin [\&] math.ru.nl

\bigskip
\noindent 2010 \emph{Mathematics Subject Classification}. 11J72, 11M06, 33C20.

\end{document}